\documentclass{conm-p-l}

\usepackage{amssymb}
\usepackage{amsmath}
\usepackage{amsfonts}
\usepackage{amsthm}
\usepackage{comment}
\usepackage{slashbox}
\usepackage{url}

\theoremstyle{plain}
\newtheorem{theorem}{Theorem}[section]
\newtheorem{lemma}[theorem]{Lemma}
\newtheorem{conj}[theorem]{Conjecture}
\newtheorem{prop}[theorem]{Proposition}
\newtheorem{cor}[theorem]{Corollary}

\theoremstyle{definition}
\newtheorem{definition}[theorem]{Definition}

\theoremstyle{remark}

\numberwithin{equation}{section}

\newcommand{\abs}[1]{\lvert#1\rvert}

\begin{document}

\title[Partial sums and convergents, II]{Which partial sums of the Taylor series for $\boldsymbol{e}$ are convergents to $\boldsymbol{e}$? (and a link to the primes 2, 5, 13, 37, 463), II } 

\author{Jonathan Sondow}
\address{ 
209 West 97th Street Apt 6F\\
New York, NY 10025\\
USA}
\email{jsondow@alumni.princeton.edu}

\author{Kyle Schalm}
\address{
P.O. Box 22017\\
Penticton, BC V2A 8L1\\
Canada}
\email{kyle.schalm@gmail.com}




\begin{abstract}
 This is an expanded version of our earlier paper. Let the $n$th partial 
 sum of the Taylor series $e = \sum_{r=0}^{\infty} 1/r!$ be $A_n/n!$, 
 and let $p_k/q_k$ be the $k$th convergent of the simple continued 
 fraction for $e$. Using a recent measure of irrationality for $e$, we 
 prove weak versions of our conjecture that only two of the partial sums 
 are convergents to $e$. A related result about the denominators $q_k$ 
 and powers of factorials is proved. We also show a surprising 
 connection between the $A_n$ and the primes $2$, $5$, $13$, $37$, 
 $463$. In the Appendix, we give a conditional proof of the conjecture, 
 assuming a second conjecture we make about the zeros of $A_n$ and $q_k$ 
 modulo powers of $2$. Tables supporting this Zeros Conjecture are 
 presented and we discuss a $2$-adic reformulation of it.
\end{abstract}

\maketitle


\section{Introduction}
 This is an expanded version of our earlier paper \cite{wpsa1}. There is 
 new material in Sections~4 and~5, and there are clarifications in several
 parts of the Appendix. Editorial problems with \cite{wpsa1} 
 were the source of many typos appearing in it. The typos are corrected 
 here.  

Based on calculations, the following conjecture was made in \cite{sondow1}.

\begin{conj} \label{conj1} \label{ps} 
Only two partial sums $A_n/n!$ of the Taylor series
\begin{equation}
\label{e}
	e = \sum_{r=0}^{\infty}\frac{1}{r!}		
\end{equation}
are convergents $p_k/q_k$ to the simple continued fraction expansion of $e$. 
\end{conj}

In the present paper, we prove some partial results toward Conjecture \ref{conj1}. One 
is that \textit{almost all the partial sums are \emph{not} convergents to $e$} 
(Corollary \ref{cor1}). The proofs do not use the known simple continued fraction expansion of $e$. Instead, the first author's \cite{sondow1} measure of irrationality for $e$ is employed --- see Lemma~\ref{lem1} part~(i).

In the Appendix, we use the continued fraction for $e$ to give a conditional proof of Conjecture \ref{conj1}, assuming 
a certain other conjecture we make about periodic behaviours of the $A_n$ and $q_k$ modulo powers of $2$ (the Zeros Conjecture). Experimental evidence for the latter is presented in the tables.

In Section 2, we prove two inequalities needed in the proofs of the main results, which are given in Section 3. The next section contains a result about denominators $q_k$ that are powers of factorials. In Section 5, we prove a surprising connection between the $A_n$ and the primes $2, 5, 13, 37, 463$.

\section{Two Lemmas} 

We establish two lemmas needed later.

\begin{lemma} \label{lem1}
Let $p/q$ be a convergent to the simple continued fraction for $e$.

\emph{(i)} If $q>1$ and $S(q)$ is the smallest positive integer such that $S(q)!$ is divisible by $q$, then 
\begin{equation}
\label{S}
      q^2 < \left(S(q)+1\right)!.
\end{equation}

\emph{(ii)} If $n! = dq$ is a multiple of $q$ with $n>0$, then 
\begin{equation}
\label{d}
	d^2 > \frac{n!}{n+1}.
\end{equation}
\end{lemma}
\begin{proof}
(i) Since $q>1$, the irrationality measure for $e$ in \cite[Theorem~1]{sondow1}, and the quadratic approximation property of convergents, give the two inequalities 
\begin{equation*}
	\frac{1}{\left(S(q)+1\right)!} < \left|e - \frac{p}{q}\right| < \frac{1}{q^2},
\end{equation*}
respectively, and \eqref{S} follows.

(ii) The inequality \eqref{d} certainly holds if $q=1$. If $q>1$, it follows from part (i), since $n! = dq$ implies $S(q)\leq n$.
\end{proof}

As an application, since $n>2$ in \eqref{d} implies $d>1$, we obtain that 
\textit{if $p/q$ is a convergent to $e$ with $q>2$, then $q$ cannot 
be a factorial}. (This is a slight improvement of \cite[Corollary~3]{sondow1}.) 

\medskip

\begin{lemma} 
\label{lem2}
For $n\ge 0$, let $s_n$ denote the $n$th partial sum of the 
series \eqref{e} for $e$, and define $A_n$ by the relations 

\begin{equation}
\label{Asum}
	\frac{A_n}{n!} = s_n := \sum_{r=0}^{n}\frac{1}{r!}. 
\end{equation}

If the greatest common divisor of $A_n$ and $n!$ is 
\begin{equation}
	d_n := \gcd(A_n,n!),
\end{equation}
then 
\begin{equation}
	d_n d_{n+1} d_{n+2} \le (n+3)!.
\end{equation}
\end{lemma}
\begin{proof}
From the recursion
$s_{n+1} = s_n + \frac{1}{(n+1)!}$
we have the relations
\begin{equation} \label{seqA}
	A_{n+1} = (n+1)A_n + 1
\end{equation}
and 
\begin{equation} \label{seqA2}
	A_{n+2} = (n+2)(n+1)A_n + (n + 3)
\end{equation}
for $n \ge 0$. Hence $\gcd(d_n,d_{n+1}) = \gcd(d_{n+1},d_{n+2}) = 1$, and 
$\gcd(d_n,d_{n+2})$ divides $(n+3)$. It follows, since $d_n$, $d_{n+1}$, $d_{n+2}$ all divide $(n+2)!$, that the product $d_n d_{n+1} d_{n+2}$ divides 
the product $(n+2)!(n+3) = (n+3)!$. This implies the result. 
\end{proof}

\section{Partial Sums vs. Convergents}
We first prove a weak form of Conjecture \ref{conj1}.

\begin{theorem} \label{thm1} \label{3cps} 
Given any three consecutive partial sums 
$s_n, s_{n+1}, s_{n+2}$ of series \eqref{e} for $e$, at most two of them 
are convergents to $e$. 
\end{theorem}

\begin{proof}
Suppose on the contrary that, for some fixed $n \geq 0$, the sums 
$s_n, s_{n+1}, s_{n+2}$ are all convergents to $e$. Then, using
Lemma~\ref{lem1} part~(ii) and the notation in Lemma \ref{lem2}, 
\begin{equation} \label{eq4}
	d_{n+j}^2 > \frac{(n+j)!}{n+j+1} \ge \frac{n!}{n+1}
\end{equation}
for $j =0,1,2$. Hence, using Lemma \ref{lem2}, 
\begin{equation}
	\left(\frac{n!}{n+1}\right)^3 < [(n+3)!]^2.
\end{equation}
This implies that $n\le13$. (\textit{Proof}. By induction, the reverse inequality 
holds for $n>13$.) But, by computation, only two of the partial sums 
$s_0, s_1, \dots, s_{15}$ are convergents to $e$ (namely, $s_1=2$ and $s_3=8/3$). This 
contradiction completes the proof.
\end{proof}

The next result is a generalization of an asymptotic version of Theorem \ref{thm1}.

\begin{theorem} \label{thm2} \label{kcps} 
For any positive integer $k$, there exists a constant $n(k)$ such that if $n\ge n(k)$, then 
among the k consecutive partial sums $s_n, s_{n+1}, \dots, s_{n+k-1}$ of 
series \eqref{e} for $e$, at most two are convergents to $e$.
\end{theorem} 

\begin{proof}
We use the notation in Lemma \ref{lem2}.

Define polynomials 
$F_1(x), F_2(x), \dots $ in $\mathbb{Z}[x]$ by the recursion 
\begin{align*}
	&F_j(x) := (x+j)F_{j-1}(x)+1, &F_1(x) := 1.
\end{align*}
 Using \eqref{seqA} and induction on $j$, we obtain the formula
\[
	A_{i+j} = (i+j)(i+j-1)\cdots(i+1)A_i + F_j(i)
\]
 for $i=0,1,\dots$ and $j =1,2,\dots$. It follows that 
\begin{align} \label{eq5}
	&\gcd(d_i,d_{i+j}) \mid F_j(i) & (i\ge0, j\ge1).
\end{align}

Now define polynomials $G_0(x), G_1(x), \dots$ in $\mathbb{Z}[x]$ recursively by
\begin{align} \label{eq6}
 	&G_j(x) := F_1(x)F_2(x) \cdots F_j(x)G_{j-1}(x), & G_0(x) := 1.
\end{align}
Since $d_i, d_{i+1},\dots,d_{i+j}$ all divide $(i+j)!$, relations \eqref{eq5} and 
\eqref{eq6} imply that the product $d_i d_{i+1}\cdots d_{i+j}$ divides the product $(i+j)!G_j(i)$, so that 
\begin{align} \label{eq7}
	&d_i d_{i+1}\cdots d_{i+j} \le (i+j)!G_j(i)	&(i\ge0, j\ge1).
\end{align}

To prove the theorem, fix $k$ and suppose on the contrary that, for infinitely many 
values of $n$, among $s_{n+1}, s_{n+2}, \dots, s_{n+k}$ there are (at least) three 
convergents to $e$ (so that $k \ge 3$), say $s_{n+a}, s_{n+b}, s_{n+c}$, where 
$1 \le a < b < c \le k$. Then, by Lemma~\ref{lem1} part~(ii), the inequalities \eqref{eq4} 
hold with $j=a, b, c$. It follows, using \eqref{eq7} with $i=n+1$ and $j=k-1$, that

\[
	\left(\frac{n!}{n+1}\right)^3 < [(n+k)!G_k(n)]^2.	
\]
Since $k$ is fixed and $G_k$ is a polynomial, Stirling's formula implies that 
$n$ is bounded. This is a contradiction, and the theorem is proved.
\end{proof}

Our final result toward Conjecture \ref{conj1} is an immediate consequence of Theorem \ref{thm2}. 

\begin{cor} \label{cor1}
Almost all partial sums of the Taylor series for $e$ are \emph{not} convergents to $e$.
\end{cor}

\section{Convergents to $e$ and Powers of Factorials}

In Section 2, we pointed out that if $p/q$ is a convergent to $e$ with $q>2$, then $q$ is not a factorial. In fact, we only need $q>1$, because the convergents to $e$ are $2/1, 3/1, 8/3, \dots$, none of which has denominator $2$.

In this section, we consider the case where $q$ is a power of a factorial. For example, the sixth convergent is $p/q=87/32$, with $q=(2!)^5$.

We obtain a curious result in which the number $e$ appears in two different ways.

\begin{prop} \label{factorial power}
Let $p/q$ be a convergent to $e$. If $q$ is a power of a factorial, say $q=(n!)^k$ with $k>0$, then $n/k < e$.
\end{prop}
\begin{proof}
Using the discussion above, we see that it suffices to prove the inequality when $n\ge 2$ and $k\ge 2$. In that case $q>1$, and so the inequality \eqref{S} holds. Since $q=(n!)^k$ divides $(nk)!$ (the quotient is a multinomial coefficient), $S(q)\leq nk$. Thus \eqref{S} implies
\begin{equation*}
       (n!)^{2k} < (nk+1)! = (nk)!(nk+1).
\end{equation*}
Using Stirling's formula
\begin{equation*}
      n! = \sqrt{2\pi n}\left(\frac{n}{e}\right)^n e^{\lambda_n} \qquad (0<\lambda_n<1),
\end{equation*}
we derive
\begin{equation*}
      (2\pi n)^k \left(\frac{n}{e}\right)^{2nk} < \sqrt{2\pi nk} \left(\frac{nk}{e}\right)^{nk} (nk+1)e.
\end{equation*}
Write this as
\begin{equation*}
      \left(\frac{n}{ke}\right)^{nk} < \frac{\sqrt{2\pi nk}(nk+1)e}{(2\pi n)^k}.
\end{equation*}
For fixed $n\geq 2$, the right side is a decreasing function of $k$, for $2\leq k < \infty$, and its value at $k=2$ is less than $1$. Therefore, $n/k<e$.
\end{proof}

\section{A Connection With The Primes 2, 5, 13, 37, 463}

In this section we show a surprising connection between the Taylor series \eqref{e} for $e$ and certain prime numbers. We use the notation in Lemma \ref{lem2}.

For $n\geq 0$, let $N_n$ denote the numerator of the $n$th partial sum $s_n$ in lowest terms, so that
\[
	N_n := \frac{A_n}{d_n}.
\]
Setting $R_n$ equal to the greatest common divisor of the reduced 
numerators $N_n$ and $N_{n+2}$ (compare relation \eqref{seqA2}),
\[
	R_n := \gcd(N_n, N_{n+2}),
\]
we find that the sequence $R_0, R_1, \ldots$ begins
\[
	1, 2, 5, \{1\}^7, 13, \{1\}^{23}, 37, \{1\}^{425}, 463, 1, 1, \ldots,
\]
where $\{1\}^k$ stands for a string of ones of length $k$. The terms $2, 5, 13, 37,$ and $463$ are primes. In fact, we prove the following result.

\begin{theorem} \label{Rseq}
The sequence $R_0, R_1, \ldots$ consists of ones and all primes in the set
\[
	P^{\ast} := \{ p \text{ prime}: 0! - 1! + 2! - 3! + 4! - \cdots + 
(-1)^{p-1}(p-1)! \equiv 0 \pmod p \}.
\]
More precisely, $R_1 = 2$, and $R_{p-3} = p$ if $p \in P^{\ast}$ is odd;
otherwise, $R_n = 1$.
\end{theorem}

 Michael Mossinghoff \cite{mossinghoff} has calculated that 2, 5, 13, 37, 
 463 are the only elements of $P^{\ast}$ less than 150 million. On the 
 other hand, at the end of this section we give a heuristic argument that 
 the set $P^{\ast}$ should be infinite, but very sparse. For this problem 
 and a related one on primes and alternating sums of factorials, see 
 \cite[B43]{guy} (where the set $P^{\ast}$ is denoted instead by $S$) and 
 \cite{zivkovic}. For $R_n$, see \cite[Sequence A124779]{oeis}.

 Before proving Theorem \ref{Rseq}, we establish two lemmas. The first 
 uses the numbers $A_n$ to give an alternate characterization of the set 
 $P^{\ast}$.

\begin{lemma} 
\label{PandA}
A prime $p$ is in $P^{\ast}$ if and only if $p$ divides $A_{p-1}$.
\end{lemma}

\begin{proof}
We show that the congruence
\[
	0! - 1! + 2! - 3! + 4! - \cdots + (-1)^{n-1}(n-1)! \equiv A_{n-1} \pmod n
\]
holds if $n > 0$. The lemma follows by setting $n$ equal to a prime $p$.

We multiply the relations \eqref{Asum} by $n!$ and replace $n$ with $n-1$. 
Re-indexing the sum, we obtain
\begin{align*}
	A_{n-1}=\sum_{r=0}^{n-1}\frac{(n-1)!}{r!}=\sum_{r=0}^{n-1}\frac{(n-1)!} {(n-1-r)!}&=\sum_{r=0}^{n-1}(n-1)(n-2)\cdots (n-r)\\
&\equiv \sum_{r=0}^{n-1}(-1)^r r! \pmod n . \qedhere
\end{align*}
\end{proof}

The next lemma gives a simple criterion for primality.

\begin{lemma} 
\label{prime}
An integer $p>4$ is prime if and only $p$ does not divide $(p-3)!$.
\end{lemma}

\begin{proof}
 The condition is clearly necessary. To prove sufficiency, we show that if 
 $p>4$ is not prime, say $p=ab$ with $b\geq a\geq 2$, then $p \mid (p-3)!$.

Since $2p-4>p\geq 2b$, we have $p-3\geq b$. In case $b>a$, we get $ab \mid (p-3)!$. In case $b=a$, we have $a\geq 3$, so $p-2a-3=a^2-2a-3=(a+1)(a-3)\geq 0$, and $p-3\geq 2a>a$ implies $(a\cdot 2a) \mid (p-3)!$. Thus, in both cases, $p \mid (p-3)!$.
\end{proof}

Now we give the \textbf{proof of Theorem \ref{Rseq}}.
\begin{proof}
We compute $N_0=1$, $N_1=2$, $N_2=5$, and $N_3=8$. Hence $R_0=1$ and $R_1=2\in P^{\ast}$.

Now fix $n>1$ and assume $R_n\neq 1$. Then $R_n$ divides both $A_n$ and 
$A_{n+2}$, but does not divide $n!$. From \eqref{seqA2} we see that $R_n \mid (n+3)$. It follows, using Lemma \ref{prime}, that $R_n=n+3$ is prime. Then Lemma \ref{PandA} implies $R_n\in P^{\ast}$.

It remains to show, conversely, that if $p\in P^{\ast}$ is odd, then 
$R_{p-3}=p$. Setting $n=p-3$, Lemma \ref{PandA} gives $p \mid A_{n+2}$. Then, as $n\geq 0$ and $p=n+3$, relation \eqref{seqA2} implies $p \mid A_n$. On the other hand, since $p>n$, the prime $p$ does not divide $n!$. It follows that $p \mid R_n$. Recalling that $R_n\neq 1$ implies $R_n$ is prime, we conclude that $R_n=p$, as desired.
\end{proof}

\noindent \textbf{ A heuristic argument that $\boldsymbol{P^{\ast}}$ is infinite but very sparse.} The following heuristics are naive. The prime 463 looks ``random,'' so a naive model might be that $0! - 1! + 2! - 3! + 4! - \cdots + (p-1)!$ is a ``random'' number modulo a prime $p$. If it is, the probability that it is divisible by $p$ would be about $1/p$. Now let's also make the hypothesis that the events are independent. Then the expected number of elements of $P^{\ast}$ which do not exceed a bound $x$ would be approximately
\[
	\#\left(P^{\ast}\cap [0,x]\right) \approx \sum_{p\leq x} 
\frac{1}{p} = \log \log x + 0.2614972128\ldots + o(1),
\]
where $p$ denotes a prime. Here the second estimate is a classical asymptotic formula of Mertens (see \cite[p. 94]{finch}). Since $\log \log x$ tends to infinity with $x$, but very slowly, the set $P^{\ast}$ should be infinite, but very sparse.

In particular, the sum of $1/p$ for primes $p$ between 463 and 150,000,000 is about $1.12$. Since this is greater than one, we might expect to find the next (i.e., the sixth) prime in $P^{\ast}$ soon.

\newpage

\section*{Appendix: Periodic Behaviour of Some Recurrence Sequences Related to $\boldsymbol{e}$, Modulo Powers of 2}
\renewcommand{\thesection}{A}
\renewcommand{\thesubsection}{A.\arabic{subsection}}
\setcounter{subsection}{0}
\setcounter{theorem}{0}

Let $A(n)/n!$ be the $n$th partial sum of series \eqref{e} for $e$, and $P(n)/Q(n)$ the $n$th 
convergent of the simple continued fraction for $e$ (note the change of notation 
from $A_n/n!$ and $p_n/q_n$ in the preceding sections). 
If $S$ denotes the integer sequence $S(0), S(1), S(2), \dots$, then we shall use 
the notation $(S \bmod M)$  
to denote the sequence $S(0)$ mod $M$, $S(1)$ mod $M$, $\dots$ .
Here ``$n$ mod $M$'' means the remainder on division of $n$ by $M$: it is a
nonnegative integer rather than an element of $\mathbb{Z}/M \mathbb{Z}$.

In this appendix, we demonstrate a relationship between Conjecture \ref{ps} and (proven and conjectured) arithmetic properties of $(A \bmod M)$ and $(Q \bmod M)$ 
for integer $M \geq 2$. 
We mainly treat the case where $M$ is a power of $2$, 
but our approach to studying $(Q \bmod {p^k})$ and $(A \bmod {p^k})$ 
should also work for odd primes $p$, with similar results.
Such investigations have not yet been undertaken.
%
%
%

The key results are Conjecture \ref{zer}, which locates the zeros of $(A \bmod M)$ 
and $(Q \bmod M)$, and Theorem \ref{mainthm}, in which we prove Conjecture \ref{ps} 
assuming Conjecture \ref{zer}. 
After that, we present some unconditional results about the periods of $(A \bmod M)$
and $(Q \bmod M)$, and discuss some consequences of them.

\medskip

The sequences $A$, $P$, and $Q$ satisfy simple linear recurrences.
Sequence $A$ satisfies recurrence \eqref{seqA} with $A(0) = 1$, and the first few values of $A(n)$ are $1$, $2$, $5$, $16$, $65$, $326$, $1957$, $13700$, $109601$, $986410$, $9864101$, $\dots$ .

Corresponding to the simple continued fraction
\begin{equation*}
	e = [2, 1, 2, 1, 1, 4, 1, 1, 6, 1, 1, 8, \dots] = [b(1), b(2), b(3), \dots]
\end{equation*}
(discovered by Euler -- see, for example, \cite{cohn}) are the recurrences
\begin{align}
	P(n) &= b(n) P(n-1) + P(n-2), & P(0)=1, P(1)=2   \label{pn} \\
	Q(n) &= b(n) Q(n-1) + Q(n-2), & Q(0)=0, Q(1)=1   \label{qn} 
\end{align}
where $b(1) = 2$ and, for $n \ge 2$,
\[
    b(n) = \begin{cases}
          2n/3 & \text{ if } 3 \mid n,\\
          1    & \text{ if } 3 \nmid n. 
          \end{cases}
\]
This correspondence, and the fact that $\gcd(P(n),Q(n))=1$, are well known by the 
general theory of continued fractions. The first few numerators $P(n)$ 
are $1$, $2$, $3$, $8$, $11$, $19$, $87$, $106$, $193$, $1264$, $1457$, $2721$, $\dots$ and 
the first few denominators $Q(n)$ are $0$, $1$, $1$, $3$, $4$, $7$, $32$, $39$, $71$, $465$, $536$, $1001$, $\dots$ .

\subsection{Main Results}

Based on calculations (portions of which are shown in Tables 1-5), we make a conjecture about the location of the zeros of 
$ ( Q \bmod M)$  and $(A \bmod M)$ for $M$ a power of $2$.
First we need a definition.

\begin{definition} For an integer $x$ and prime $p$, let 
\[
  [x]_p = \sup\{p^k:p^k \mid x \text{ and } k \in \mathbb{N}\}
\]
denote the \emph{$p$-factor} of $x$. Note that $[0]_p = \infty$, $[xy]_p = [x]_p [y]_p$, and $1 \le [x]_p \le \abs{x}$ for $x \neq 0$.
\end{definition}

\begin{conj}[Zeros Conjecture]
\label{zer}  
For each $n \ge 0$,
\begin{align*}
 [Q(3n)]_2 &\le 4[n(n+2)]_2,                    \tag{i}\\
 [Q(3n+1)]_2 &\le 2[n+1]_2,                     \tag{ii}\\
 [Q(3n+2)]_2 &= 1,                              \tag{iii}\\
 [A(n)]_2 &\le (n+1)^2.                         \tag{iv}
\end{align*}
\end{conj}

Statement (iii) is easily proven, but it is placed with the others for 
harmony. Statement (iv) is somewhat arbitrary in form and can probably be 
strengthened, but it is difficult to guess the exact truth in this case. By 
contrast, we believe that equality holds in (i) and (ii) infinitely often.

Proof of (iii): using \eqref{qn} twice,
\begin{align*}
   Q(3n+2) &= Q(3n+1) + Q(3n) \\
           &= 2Q(3n) + Q(3n-1).
\end{align*}
Since $Q(2)$ is odd, it follows by induction that $Q(3n+2)$ is odd for $n \geq 0$.

\medskip

Conjecture \ref{zer} implies information about the zeros of $(Q \bmod M)$  as follows: 
\begin{align*}
\text{ if }&& Q(3n)   &\equiv 0 \pmod {2^k} &\text{ then }&& n &\equiv 0,-2 \pmod {2^{k-3}} \\
\text{ if }&& Q(3n+1) &\equiv 0 \pmod {2^k} &\text{ then }&& n &\equiv -1   \pmod {2^{k-1}} 
\end{align*}
for $k \geq 3$ and $k \geq 1$, respectively, while $Q(3n+2)$ is always nonzero modulo an even number.
The connection between (iv) and the location of the zeros of $(A \bmod M)$ is a little fuzzy here. It is clarified somewhat in part (iv) of Conjecture \ref{adic}.

\begin{lemma} 
\label{nmlem}
Let $n>1$ be an integer and $N$ be the unique integer for which $3N \le n < 3(N+1)$. If $m$ is a positive integer such that $Q(n) \le m!$, then $N < m$ and $n < 3m$. 
\end{lemma}

\begin{proof}
First verify the cases $n=2$ and $n=3$ directly. 

Next suppose that $n=3N$ for some $N>1$. Using \eqref{qn} in the form $Q(n) > b(n) Q(n-1)$ (since $Q(n-2) > 0$ for $n>2$), we have $Q(n) = Q(3N) > 2N Q(3N-1) > 2N Q(3N-2) > 2N Q(3N-3)$. Since $Q(3)=3$, it follows that $Q(3N) > 2N \cdot 2(N-1) \cdot 2(N-2) \cdots 2(2) \cdot Q(3) = (3/2) 2^N N! > N!$. Thus if $Q(3N) \le m!$ then $N < m$. 

Finally suppose that $n=3N+1$ or $n=3N+2$ for some $N \ge 1$. If $Q(n) \le m!$ then the same conclusion holds, because $Q(3N) < Q(n)$. So in all cases, $Q(n) \le m!$ implies $N < m$.

From $n < 3(N+1)$ we also have $n < 3m$ since $N+1 \le m$, and this proves 
the result.
\end{proof}

\begin{theorem}   
\label{mainthm} 
Conjecture \ref{zer} implies Conjecture \ref{ps}.
\end{theorem}

\begin{proof}
Assume that a partial sum of series \eqref{e} is a convergent to $e$, say $A(m)/m! = P(n)/Q(n)$. Write this as 
\begin{equation}
\label{eq1}
A(m) Q(n) = m! P(n).
\end{equation}

The general strategy is as follows: by examining how the $2$-factors of $Q(n)$, $A(m)$, and $m!$ grow, we show that \eqref{eq1} has no solution except for some small values of $m$ and $n$. Specifically, $[A(m)]_2$ and $[Q(n)]_2$ grow slowly whereas $[m!]_2$ grows quickly, so we should expect that
\begin{equation}
\label{eq2}
    [A(m) Q(n)]_2 < [m! P(n)]_2 
\end{equation}
unless $m$ and $n$ are sufficiently small. Since \eqref{eq2} contradicts \eqref{eq1}, we will have shown that \eqref{eq1} has no solutions except possibly those permitted by the exceptions to \eqref{eq2}, which we check by computer.

We will need some preliminary inequalities. Assume that $n>1$ and let $N$ be as in Lemma \ref{nmlem}. 

\begin{itemize}

\item Observe that $4[N(N+2)]_2 \le \max\{8[N]_2,8[N+2]_2\}$ since $\gcd(N,N+2) \le 2$. Then Conjecture \ref{zer} (i)-(iii) imply that
\[
	[Q(n)]_2 \le \max\{8[N]_2, 8[N+2]_2, 2[N+1]_2, 1\} \le 8(N+2)
\]
since $[x]_2 \le x$. 

\item Since $\gcd(P(n), Q(n)) = 1$, there are no solutions to \eqref{eq1} if 
$Q(n) \nmid m!$. So for \eqref{eq1} to hold, it must be that $Q(n) \mid m!$ and in particular $Q(n) \leq m!$. From this we can apply Lemma \ref{nmlem} to deduce that $N \le m-1$.

\item For every positive integer $m$, we have $[m!]_2 \ge 2^m/(m+1)$. This follows from the formula $\text{ord}_p(m!) = (m-\sigma_p(m))/(p-1)$ 
(see \cite[p. 79]{koblitz}), where $p$ is any prime, 
$\text{ord}_p(x) := \log_p([x]_p)$, 
and $\sigma_p(m)$ is the sum of the base-$p$ digits of $m$: take $p=2$ and use $\sigma_2(m) \le \log_2{(m+1)}$. If $m > 20$, then $2^m > 8(m+1)^4$ and thus $[m!]_2 > 8(m+1)^3$.

\item Trivially, $1 \le [P(n)]_2$.

\end{itemize}

For $m>20$, making use of Conjecture \ref{zer} (iv) and the above inequalities, we get 
\[
   [A(m) Q(n)]_2 \le (m+1)^2 \cdot 8(N+2) 
               \le (m+1)^2 \cdot 8(m+1) 
                 < [m!]_2
	       \le [m! P(n)]_2.
\]
Thus \eqref{eq2} holds for $m > 20$ and $n>1$. There are a finite number of remaining cases, since $m \le 20$ implies, by Lemma \ref{nmlem}, that $n < 60$. We verified by computer that \eqref{eq1} has no solution for these cases, with the two exceptions $m=n=1$ and $m=n=3$, corresponding to the convergents $2/1$ and $8/3$.
\end{proof}

\subsection{Periodicity}

In this section we relate some observations about the (actual or apparent) periodicity of $(A \bmod M)$ and $(Q \bmod M)$ for a positive integer $M$.
While independent of the preceding results, they nevertheless seem worth mentioning. 
An eventual proof of Conjecture \ref{zer} would likely make use of such results.

\begin{prop} \label{A_period}
For any integer $M>0$, the sequence 
$( A \bmod M )$ 
is periodic with period exactly $M$. 
\end{prop}

\begin{proof} 

Since $A(M) = M A(M-1) + 1$, we have $A(M) \equiv 1 \equiv A(0) \pmod M$, and induction on $n$ using \eqref{seqA} gives $A(M+n) \equiv A(n) \pmod M$ for $n \ge 0$.
This last congruence is equivalent to saying that a period $P$ exists and $P | M$.

Next we show that $M | P$. The definition of $P$ gives $A(P) \equiv A(0) \pmod M$, so 
\begin{align*}
A(P+1) &= (P+1)A(P) + 1 \\
        &\equiv (P+1)A(0) + 1        \pmod M \\
        &= A(0)+1 + PA(0) \\
        &= A(1) + P.
\end{align*}
But the definition of $P$ also gives $A(P+1) \equiv A(1) \pmod M$, so  $P \equiv 0 \pmod M$.

Since $P | M$ and $M | P$, we conclude that $P=M$. 
\end{proof} 

\noindent 
\textbf{Remark.} $\,$  This result generalizes to the recurrence 
$S(n) = nS(n-1) + S(0)$ with an arbitrary integer initial value $S(0)$, and 
the result in this case is that 
the period of $(S \bmod M)$  is $M/\gcd(M,S(0))$. \\

 
One would like to prove a similar result for $Q$; here we have only met with partial success. Following are a proof that a period exists, and a conjecture about the value of that period.

\begin{prop} \label{prop2}
For any integer $M>0$, the sequence 
$( Q \bmod M )$
is periodic, with period at most $3M^3$. 
\end{prop}

\begin{proof} 
We mimic the proof in \cite[Theorem 1]{wall}, which was applied there only to the Fibonacci sequence. It is based on the Pigeonhole Principle.

Neglecting the initial term, the sequence $(b \mod M)$
is periodic with period dividing $3M$ (meaning $b(n) \equiv b(n+3M) \pmod M$ 
as long as $n > 1$). So if there exist integers $h=h(M)$ and $i=i(M)$ 
with $i>h$ such that $i \equiv h \pmod {3M}$ and $Q(i) \equiv Q(h),\ \ 
Q(i+1) \equiv Q(h+1) \pmod M$, then by applying the recurrence \eqref{qn} 
repeatedly, we have $Q(i+n) \equiv Q(h+n) \pmod M$ for $n \ge 0$. There are 
only $3M^3$ possible values of the triple ($n \bmod {3M}$, $Q(n) \bmod M$, 
$Q(n+1) \bmod M$), so they must repeat eventually and therefore 
such an $h$ and $i$ exist. 

To show that we can take $h=0$, reverse the recurrence to $Q(n-2) = Q(n) - b(n) Q(n-1)$. Apply it repeatedly, concluding that $Q(0) \equiv Q(i-h) \pmod M$.
\end{proof}

\begin{definition}
For $i=0,1,2$, let $Q_i$ be the subsequence of $Q$ consisting of every third element beginning with the $i$\textsuperscript{th} one, that is, $Q_i(n) = Q(3n+i)$. 
\end{definition}

The periodicity of $(Q \bmod M)$ obviously implies the periodicity of all three $(Q_i \bmod M)$, and vice versa. 

\begin{conj}[Period Conjecture] \label{per}
\ \\
{\rm (a)} If $M>1$ is odd, then for $i=0,1,2$, the period of $(Q_i \bmod M)$ equals $2M$. \\
{\rm (b)} 
If $M>0$ is even, then for $i=0,1,2$, the period of $(Q_i \bmod M)$ divides $M$.
\end{conj}

This conjecture has been verified numerically for $M \leq 1000$.
For $M$ a power of $2$, some of these calculations are shown in Tables 1-3, and a more exact conjecture is given in the last column of Table \ref{table_qi}. 

\subsection{A Possible 2-adic Approach}

In this section we reformulate some of the preceding results in the language of $p$-adic analysis.
Let $p$ be prime, let ${\mathbb{Z}}_p$ denote 
the $p$-adic integers, and let $|\cdot|_p$ be the usual $p$-adic absolute value on ${\mathbb{Z}}_p$ (so $|x|_p = [x]_p^{-1}$ for $x \in {\mathbb{Z}}$). In particular, we consider $p=2$ in what follows.

\begin{lemma} \label{parity_lemma}
If $n$ is odd, then $A(n) \not\equiv A(n+2^k) \pmod {2^{k+1}}$ for all $k \geq 0$.
\end{lemma}
The proof relies on Proposition \ref{A_period} and elementary arguments. We omit the details for the sake of brevity.

\begin{prop} \label{A_zero}\ \\
{\rm (i)} The sequence $A$ extends uniquely to a continuous 
function $\tilde{A}:\mathbb{Z}_2 \rightarrow \mathbb{Z}_2$ 
(so $\tilde{A}(n) = A(n)$ for $n = 0, 1, 2, \dots$). \\
{\rm (ii)} For each $k \geq 1$, 
the interval $[0,2^k)$ contains a unique zero $c_k$ of 
$A \bmod 2^k$
(that is, $A(c_k) \equiv 0 \pmod {2^k}$). 
See Table 6 for the first few $c_k$. \\
{\rm (iii)} The function $\tilde{A}$ has the unique zero 
\[
  c = \lim_{k \to \infty } c_k  
    = 11001110010100010100110001\ldots \in \mathbb{Z}_2 
\]
where the limit is taken in $\mathbb{Z}_2$. For $c$ see 
{\rm \cite[Sequences A127014 and A127015]{oeis}}.\\
{\rm (iv)} For each $n \in \mathbb{Z}_2$, we have $|\tilde{A}(n)|_2 = |n-c|_2$.
\end{prop}

\begin{proof}
(i) This is a simple consequence of Proposition \ref{A_period}. Since $m \equiv n \pmod {2^k}$ implies $A(m) \equiv A(n) \pmod {2^k}$, it follows that $|A(m) - A(n)|_2 \leq |m-n|_2$. 

(ii) We use induction on $k$.
For $k=1$, the congruence $A(n) \equiv 0 \pmod 2$ has the unique solution $n \equiv c_1 \equiv 1 \pmod 2$. (Note for later that $c_k$ is odd, since $c_k \equiv c_1 \pmod 2$.)
Now assume that $A(n) \equiv 0 \pmod {2^k}$ has the unique solution $n \equiv c_k \pmod {2^k}$.
Let us solve $A(n) \equiv 0 \pmod {2^{k+1}}$ for $n$. Reducing modulo $2^k$, we get $A(n) \equiv 0 \pmod {2^k}$, which by the inductive hypothesis implies $n \equiv c_k \pmod {2^k}$.
This corresponds to the two possible solutions $n \equiv c_k \pmod {2^{k+1}}$ and $n \equiv c_k+2^k \pmod {2^{k+1}}$; let $f = A(c_k)$ and let $g = A(c_k+2^k)$.
Then (using Prop. \ref{A_period} with $M = 2^k$) we have $f \equiv g \equiv 0 \pmod {2^k}$, which implies that each of $f$ and $g$ is congruent to $0$ or $2^k$ modulo $2^{k+1}$.
But Lemma \ref{parity_lemma} implies that $f \not\equiv g \pmod {2^{k+1}}$, so one of them must be zero, and one must be nonzero.
Hence a zero of 
$( A\bmod 2^{k+1} )$
exists and is unique, up to translation by a multiple of the period $2^{k+1}$
 (again by Proposition \ref{A_period}, with $M=2^{k+1}$).

(iii) The limit exists since $c_{k+1} \equiv c_k \pmod {2^k}$, and is unique since there is a unique zero of $(A \bmod 2^{k} )$  for each $k$.

(iv) This is a special case of the stronger equality
$|\tilde{A}(n)-\tilde{A}(m)|_2 = |n-m|_2$, which holds if $m$ and $n$ are not both even.
The proof of the $\leq$ direction is in the argument for part (i); the proof of the $\geq$ direction requires Lemma \ref{parity_lemma}. We omit the details.
\end{proof}

\begin{cor} For all $k \geq 1$,
\[ 
c_{k+1} =
\begin{cases}
  c_k       & \text{ if } 2^{k+1}|A(c_k), \\
  c_k + 2^k & \text{ otherwise. }
\end{cases}
\]
\end{cor}
\begin{proof}
This is immediate from Proposition \ref{A_zero} part (ii) and its proof.
\end{proof}

If Conjecture \ref{per} is true, then similarly $Q_i$ extends uniquely to a continuous function $\tilde{Q}_i:\mathbb{Z}_2 \rightarrow \mathbb{Z}_2$ for $i=0,1,2$. 
%
%
In that case, we can replace Conjecture \ref{zer} with the slightly stronger 

\begin{conj} \label{adic} 
For all $n \in \mathbb{Z}_2$ and $k \geq 1$,
\begin{align*}
|\tilde{Q}_0(n)|_2 &\ge |4n(n+2)|_2,         \tag{i}\\
|\tilde{Q}_1(n)|_2 &\ge |2(n+1)|_2,          \tag{ii}\\
|\tilde{Q}_2(n)|_2 &= 1,                     \tag{iii}\\
|c-c_k|_2 &\geq 2^{-2k}.                     \tag{iv}
\end{align*}
\end{conj}

For $0 \le n \in \mathbb{Z}$, statements (i)-(iii) are equivalent to statements (i)-(iii) of Conjecture \ref{zer}. On the other hand, \ref{zer}(i)-(iii) and Conjecture \ref{per} imply \ref{adic}(i)-(iii) for all $n \in \mathbb{Z}_2$, by continuity.

It is not immediately obvious that statement \ref{adic}(iv) implies statement \ref{zer}(iv), but it does. The proof makes use of part (iv) of Proposition \ref{A_zero}, among other things.
Statement \ref{adic}(iv) is also equivalent to the statement that there are never more consecutive zeros in the $2$-adic expansion of $c$ than the number of digits preceding those zeros.
As far as progress toward this conjecture goes, we lack a 
description of $c$ at this time other than as a sequence of digits computed 
by brute force (as illustrated in Table 6). In particular, there is no obvious
pattern to the distribution of ones and zeros in its $2$-adic expansion. 

\medskip

\noindent 
\textbf{Remarks.}
 
1. Concerning the extension of $A$ to $\mathbb{Z}_2$, we have gained something unexpected. The extension of $Q_i$ to negative integers can be effected directly from the defining recurrence, and this extension agrees with the one obtained via $\tilde{Q}_i$:
\[
      Q_i(-n) = \tilde{Q}_i(-n) = \lim_{k \to \infty} Q_i(2^k-n).
\]
But in the case of $A$, we cannot use the recurrence because 
\[
      A(0) = 0 \cdot A(-1) + 1
\]
cannot be solved for $A(-1)$. It seems as if $A(-1)$ is a free parameter that allows us to extend $A$ to $-\mathbb{N}$ in any number of equally natural ways. But in fact because of the existence of $\tilde{A}$, we see that
\[
     A(-1) = \tilde{A}(-1) = \lim_{k \to \infty} A(2^k-1) = 0011110100110010 \dots \in \mathbb{Z}_2 \setminus \mathbb{Z}
\]
is a privileged choice.

Like $c$, the number $\tilde{A}(-1)$ is an interesting-looking constant that would enjoy being studied further.



2. The (hopeful) point of the $p$-adic approach is to understand $A$ and $Q$ by studying $\tilde{A}$ and $\tilde{Q}_i$ using methods of $p$-adic analysis. Are $\tilde{A}$ and $\tilde{Q}_i$ differentiable? Are they analytic? Is it possible to represent them by power series or integrals? Can iterative root-finding methods be used to compute $c$ quickly?

\section*{Acknowledgements}

The computer algebra systems SAGE \cite{sage} and PARI/GP \cite{pari} were used to do calculations. 
We are grateful to Kevin Buzzard, Kieren MacMillan,
and Wadim Zudilin for suggestions on Section 5, and to David Loeffler \cite{loeffler} and Sergey Zlobin \cite{zlobin} for early computations of $P^{\ast}$ up to 2 million and 4 million, respectively. We thank MacMillan also for his help typesetting the Appendix.
The first author is grateful to the Department of Mathematics at Keio
University in Japan for its hospitality during his visit in November, 2006,
when part of Section 5 was written.

\newpage
\vspace{2em}
\begin{center}
\Large \textbf{Tables}
\end{center}

\begin{table}[ht]
\begin{center}
\caption{$Q_0(n)$ mod $2^k$}
\begin{tabular}{|l|r|r|r|r|r|r|r|r|r|r|r|r|r|r|r|r|r|}
\hline 
\backslashbox{$k$}{$n$} & 0 &1 &2 & 3& 4& 5& 6& 7& 8& 9& 10& 11& 12& 13& 14& 15& period\\
\hline
1 & 0& 1& 0&  1&  0&  1&  0&  1&  0&  1&  0&  1&  0&  1&  0& 1&   2 \\
2 & 0& 3& 0&  1&  0&  3&  0&  1&  0&  3&  0&  1&  0&  3&  0& 1&   4 \\
3 & 0& 3& 0&  1&  0&  3&  0&  1&  0&  3&  0&  1&  0&  3&  0& 1&   4 \\
4 & 0& 3& 0&  1&  0&  3&  0&  1&  0&  3&  0&  1&  0&  3&  0& 1&   4 \\
5 & 0& 3& 0&  17& 0&  3&  0&  1&  0&  3&  0&  17& 0&  3&  0& 1&   8 \\
6 & 0& 3& 32& 17& 32& 35& 0&  1&  0&  35& 32& 17& 32& 3&  0& 1&   16\\
7 & 0& 3& 32& 81& 96& 99& 64& 65& 64& 35& 96& 81& 32& 67& 0& 1&   32\\
\hline
[$Q_0(n)]_2$ 
  & -& 1 & 32& 1 & 32& 1 & 64& 1& 64& 1 & 32& 1 & 32& 1 & 128& 1 & -\\
\hline
\end{tabular}
\end{center}
\end{table}

\begin{table}[ht]
\begin{center}
\caption{$Q_1(n)$ mod $2^k$}
\begin{tabular}{|l|r|r|r|r|r|r|r|r|r|r|r|r|r|r|r|r|r|}
\hline 
\backslashbox{$k$}{$n$} & 0 &1 &2 & 3& 4& 5& 6& 7& 8& 9& 10& 11& 12& 13& 14& 15& period\\
\hline
1 & 1& 0& 1&  0&  1&  0&  1&  0&  1&  0&  1&  0&  1&  0&  1&  0&   2\\
2 & 1& 0& 3&  0&  1&  0&  3&  0&  1&  0&  3&  0&  1&  0&  3&  0&   4\\
3 & 1& 4& 7&  0&  1&  4&  7&  0&  1&  4&  7&  0&  1&  4&  7&  0&   4\\
4 & 1& 4& 7&  8&  9&  12& 15& 0&  1&  4&  7&  8&  9&  12& 15& 0&   8\\
5 & 1& 4& 7&  24& 9&  12& 15& 16& 17& 20& 23& 8&  25& 28& 31& 0&   16\\
6 & 1& 4& 39& 24& 9&  12& 47& 48& 49& 20& 23& 8&  57& 28& 31& 32&  32 \\
7 & 1& 4& 39& 24& 73& 12& 47& 48& 49& 20& 23& 72& 57& 28& 95& 96&  64 \\
\hline
[$Q_1(n)]_2$ 
  & 1& 4 & 1& 8 & 1&  4 & 1&  16& 1&  4 & 1&  8 & 1&  4 & 1&  32 & -\\
\hline
\end{tabular}
\end{center}
\end{table}

\begin{table}[ht]
\begin{center}
\caption{$Q_2(n)$ mod $2^k$}
\begin{tabular}{|l|r|r|r|r|r|r|r|r|r|r|r|r|r|r|r|r|r|}
\hline 
\backslashbox{$k$}{$n$} & 0 &1 &2 & 3& 4& 5& 6& 7& 8& 9& 10& 11& 12& 13& 14& 15& period\\
\hline
1 & 1& 1& 1&  1&   1&  1&   1&   1&   1&   1&  1&   1&  1&  1&  1&  1&   1 \\
2 & 1& 3& 3&  1&   1&  3&   3&   1&   1&   3&  3&   1&  1&  3&  3&  1&   4 \\
3 & 1& 7& 7&  1&   1&  7&   7&   1&   1&   7&  7&   1&  1&  7&  7&  1&   4 \\
4 & 1& 7& 7&  9&   9&  15&  15&  1&   1&   7&  7&   9&  9&  15& 15& 1&   8 \\
5 & 1& 7& 7&  9&   9&  15&  15&  17&  17&  23& 23&  25& 25& 31& 31& 1&   16 \\
6 & 1& 7& 7&  41&  41& 47&  47&  49&  49&  55& 55&  25& 25& 31& 31& 33&  32\\
7 & 1& 7& 71& 105& 41& 111& 111& 113& 113& 55& 119& 25& 89& 95& 95& 97&  64 \\
\hline
[$Q_2(n)]_2$ 
  & 1& 1 & 1& 1 & 1&  1 & 1&  1& 1&  1 & 1&  1 & 1&  1 & 1&  1 & -\\
\hline
\end{tabular}
\end{center}
\end{table}

\begin{table}[ht]  
\begin{center}
\caption{Period of $( Q_{i} \bmod 2^{k})$}
\label{table_qi}
\begin{tabular}{|l|r|r|r|r|r|r|r|r|r|r|l|} 
\hline 
\backslashbox{seq.}{$k$} & 1& 2& 3& 4& 5& 6& 7& 8& 9& 10  & conjecture\\
\hline
$Q_0$ & 2& 4& 4& 4&  8& 16& 32&  64& 128& 256     & $2^{k-2}$ for $k>3$   \\
$Q_1$ & 2& 4& 4& 8& 16& 32& 64& 128& 256& 512     & $2^{k-1}$ for $k>2$   \\
$Q_2$ & 1& 4& 4& 8& 16& 32& 64& 128& 256& 512     & $2^{k-1}$ for $k>2$   \\
\hline
\end{tabular}
\end{center}
\end{table}

\begin{table}[ht]
\begin{center}
\caption{$A(n)$ mod $2^k$}
\begin{tabular}{|l|r|r|r|r|r|r|r|r|r|r|r|r|r|r|r|r|r|}
\hline 
\backslashbox{$k$}{$n$} & 0 &1 &2 & 3& 4& 5& 6& 7& 8& 9& 10& 11& 12& 13& 14& 15& period\\
\hline
1 & 1& 0& 1& 0&  1&  0&  1&  0& 1&  0&  1&  0&  1&  0&  1& 0&   2 \\
2 & 1& 2& 1& 0&  1&  2&  1&  0& 1&  2&  1&  0&  1&  2&  1& 0&   4 \\
3 & 1& 2& 5& 0&  1&  6&  5&  4& 1&  2&  5&  0&  1&  6&  5& 4&   8 \\
4 & 1& 2& 5& 0&  1&  6&  5&  4& 1&  10& 5&  8&  1&  14& 5& 12&  16 \\
5 & 1& 2& 5& 16& 1&  6&  5&  4& 1&  10& 5&  24& 1&  14& 5& 12&  32 \\
6 & 1& 2& 5& 16& 1&  6&  37& 4& 33& 42& 37& 24& 33& 46& 5& 12&  64\\
7 & 1& 2& 5& 16& 65& 70& 37& 4& 33& 42& 37& 24& 33& 46& 5& 76&  128\\
\hline
[$A(n)]_2$ 
  & 1& 2 & 1& 16 & 1& 2& 1&  4& 1&  2 & 1&  8 & 1&  2 & 1&  4 & -\\
\hline
\end{tabular}
\end{center}
\end{table}

\begin{table}[tb]
\begin{center}
\caption{$c_k = $ smallest $n$ such that $A(n)$ is divisible by $2^k$} 
\begin{tabular}{|r|r|l|l|} 
\hline 
$k$ & $c_k$ in decimal & $c_k$ in $2$-adic notation  & $c_k - c_{k-1}$ \\
    & notation         & (reverse binary)            & \\
\hline
1 & 1 & 1 & -\\
2 & 3 & 11 & $2^1$ \\
3 & 3 & 11 & $0$ \\
4 & 3 & 11 & $0$ \\
5 & 19 & 11001 & $2^4$ \\
6 & 51 & 110011 & $2^5$ \\
7 & 115 & 1100111 & $2^6$ \\
8 & 115 & 1100111 & $0$ \\
9 & 115 & 1100111 & $0$ \\
10 & 627 & 1100111001 & $2^9$ \\
11 & 627 & 1100111001 & $0$ \\
12 & 2675 & 110011100101 & $2^{11}$ \\
13 & 2675 & 110011100101 & $0$ \\
14 & 2675 & 110011100101 & $0$ \\
15 & 2675 & 110011100101 & $0$ \\
16 & 35443 & 1100111001010001 & $2^{15}$ \\ 
17 & 35443 & 1100111001010001 & $0$ \\
18 & 166515 & 110011100101000101 & $2^{17}$ \\
19 & 166515 & 110011100101000101 & $0$ \\
20 & 166515 & 110011100101000101 & $0$ \\
21 & 1215091 & 110011100101000101001 & $2^{20}$ \\
22 & 3312243 & 1100111001010001010011 & $2^{21}$ \\
\hline
\end{tabular}
\end{center}
\end{table}

\clearpage

\bibliographystyle{amsplain}

\end{document}